\scrollmode

\documentclass{amsart}
\usepackage{color}
\usepackage{amssymb}
\usepackage[all]{xy}
\usepackage{hyperref}
\usepackage{verbatim}

\theoremstyle{plain}
\newtheorem{prop}{Proposition}
\newtheorem{thm}[prop]{Theorem}
\newtheorem{cor}[prop]{Corollary}
\newtheorem{lem}[prop]{Lemma}
\newtheorem{fact}[prop]{Fact}

\newtheorem*{thmA}{Theorem A}

\usepackage{amsfonts}
\usepackage{amsmath}
\usepackage{amsthm}
\usepackage{textcomp}
\usepackage{amscd}

\theoremstyle{definition}
\newtheorem*{defi}{Definition}

\theoremstyle{remark}

\newtheorem{rem}[prop]{Remark}
\newtheorem{example}[prop]{Example}

\numberwithin{prop}{section}
\numberwithin{equation}{section}

\DeclareMathOperator{\Gal}{Gal}

\DeclareMathOperator{\Hom}{Hom}
\DeclareMathOperator{\Aut}{Aut}

\DeclareMathOperator{\Zen}{Z}
\DeclareMathOperator{\image}{im}
\DeclareMathOperator{\kernel}{ker}

\DeclareMathOperator{\bra}{Br}
\DeclareMathOperator{\cohdim}{cd}
\DeclareMathOperator{\rank}{rk}
\DeclareMathOperator{\res}{res}

\DeclareMathOperator{\spa}{Span}

\DeclareMathOperator{\Cent}{C}

\newcommand{\Z}{\mathbb{Z}}
\newcommand{\Q}{\mathbb{Q}}
\newcommand{\F}{\mathbb{F}}
\newcommand{\N}{\mathbb{N}}

\newcommand{\argu}{\hbox to 7truept{\hrulefill}}

\author{S. K. Chebolu}
\author{J. Min\'a\v{c}}
\author{C. Quadrelli}
\title[Detecting Fast solvability of equations]
      {Detecting Fast solvability of equations via small powerful Galois groups}
\date{\today}

\thanks{The first author is partially supported by NSA grant H98230-13-1-0238
	and the second author from NSERC grant RO37OA1OO6}

\address{Department of Mathematics, Illinois State University\\
Campus box 4520, 61790 Normal IL, USA}
\email{schebol@ilstu.edu}
\address{Department of Mathematics, University of Western Ontario\\
Middlesex College, N6A5B7 London ON, Canada}
\email{minac@uwo.ca}
\address{Dipartimento di Matematica, Universit\`a di Milano-Bicocca\\
Ed. U5, Via R.Cozzi 53, 20125 Milano, Italy}
\email{c.quadrelli1@campus.unimib.it}

\begin{document}

\begin{abstract}
Fix an odd prime $p$, and let $F$ be a field containing a primitive $p$th root of unity.
It is known that a $p$-rigid field $F$ is characterized by the property that the Galois group $G_F(p)$
of the maximal $p$-extension $F(p)/F$ is a solvable group.
We give a new characterization of $p$-rigidity which says that a field $F$ is $p$-rigid precisely
when two fundamental canonical quotients of the absolute Galois groups coincide.
This condition is further related to analytic $p$-adic groups and to some Galois modules.
When $F$ is $p$-rigid, we also show that it is possible to solve for the roots of any irreducible polynomials in $F[X]$ whose splitting field over $F$ has a $p$-power degree via non-nested radicals.
We provide new direct proofs for hereditary $p$-rigidity,
together with some characterizations for $G_F(p)$ -- including a complete description for such a group
and for the action of it on $F(p)$ -- in the case $F$ is $p$-rigid.
\end{abstract}


\dedicatory{ To Professors Tsit-Yuen Lam and Helmut Koch
with admiration and respect.}

\maketitle

\section{Introduction}\label{sec:1intro}
The problem of solving algebraic equations by radicals has a long and rich history
which dates back to the 7th century when the Indian mathematician Brahmagupta
obtained the famous quadratic formula.
After the Italian mathematicians Niccol\`o Tartaglia and Girolamo Cardano
obtained the solution of the cubic equation
in the 16th century, mathematicians naturally wondered
whether it is possible to solve equations of any degree by radicals.
\'Evariste Galois, in his theory of equations, gave an elegant answer in the 19th century.
It is possible to solve an equation by radicals,
provided the Galois group of the underlying equation is a solvable group.
An important consequence is the result on the insolvability of general algebraic equations
of degree 5 and above by radicals.
Since every finite $p$-group is a solvable group, we know that every irreducible polynomial in $F[X]$
whose splitting field over $F$ has $p$-power degree, is solvable by radicals. In this paper,
we will show that if the underlying field $F$ is $p$-rigid (see definition below), then it is possible to do even better:
we can ``fast-solve" for the roots. That is, we can solve for the roots of these irreducible polynomials via non-nested radicals, i.e., elements of the type $\sqrt[n]{a}$
with $a\in F\smallsetminus\{0\}$.
(An element of the form $\sqrt[n]{a+\sqrt[m]{b}}$, with $a,b\in F\smallsetminus\{0\}$ and $n,m>1$
is, for instance, nested.)
This improves the following result: the Galois group of the maximal $p$-extension $F(p)/F$ is a solvable group
if and only if the field $F$ is $p$-rigid (proved first in \cite{englerkoen}).

We make the blanket assumption that $p$ is an odd prime, and
all fields in this paper contain a primitive $p$th root of unity -- unless explicitly stated otherwise.
Let $F$ be such a field and let $F^p$ denote the collection of elements in $F$ that are $p$th powers.
An element $a$ in $ F\smallsetminus F^p$ is said to be $p$-rigid if the image of the norm map $F(\sqrt[p]{a})\rightarrow F$
is contained in $\bigcup_{k=0}^{p-1}a^kF^p$.
We say that $F$ is $p$-{\it rigid} if all of the elements of $F\smallsetminus F^p$ are $p$-rigid.
The notion of $p$-rigidity was introduced by K. Szymiczek in \cite[Ch. III, \S 2]{szy},
and it was developed and thoroughly studied first in the case of $p=2$, and then in the case of $p$ odd.

For $p=2$, the definition of $2$-rigidity depends on the behavior of certain quadratic forms.
The consequences of $2$-rigidity were studied in several papers including \cite{ware1,ware2,jacob, jacobeare, AdGaKaMi, leepsmith}.
Today many results about 2-rigid fields are known, and these fields are relatively well understood.

For $p$ odd, the study of $p$-rigid fields was developed by Ware in \cite{ware},
and later on by others (see \cite{Ktheoryefrat} and \cite{secret} for some highlights on the history of $p$-rigidity).
In \cite{ware}, Ware introduced a different notion of rigidity called {\it hereditary $p$-rigidity}.
A field $F$ is said to be hereditarily $p$-rigid if every subextension of the maximal $p$-extension
$F(p)/F$ is $p$-rigid. As Ware pointed out, to  conclude that $F$ is hereditary $p$-rigid, it is enough to check that each finite extension $K$ of $F$ is $p$-rigid.
Ware also gave a Galois-theoretic description of hereditarily $p$-rigid fields.
In \cite{englerkoen}, A. Engler and J. Koenigsmann showed that
$p$-rigidity implies hereditary $p$-rigidity.

In this paper we  establish some new characterizations and deeper connections for $p$-rigid fields.
 Associated to  a field $F$, we now introduce some important field extensions.  Let $F^{(2)}=F(\sqrt[p]{F})$.
Let $F^{\{3\}}$ denote the compositum of all Galois extensions $K/F^{(2)}$ of degree $p$. Similarly, let  $F^{(3)}$
denote the compositum of all Galois extensions $K/F^{(2)}$ of degree $p$ for which $K/F$ is also Galois. Finally, let $F(p)$ denote the
 compositum of all  Galois extensions $K/F$ that are of degree a power of $p$.
Our main theorem then states:

\begin{thmA}
Let $p$ be an odd prime, and let $F$ be a field containing a primitive $p$-th root of unity.
 $F$ is $p$-rigid if and only if $F^{(3)}=F^{\{3\}}$.
\end{thmA}

It is worth pointing out that the fields $F^{(3)}$ and $F^{\{3\}}$ play an important role in studying the arithmetic and Galois cohomology of fields.
For instance, in \cite{MinacSpira} it is shown that $\Gal(F^{(3)}/F)$ in the case when $p=2$ determines
essentially the Witt ring of quadratic forms. More recently, in  \cite{cem} it is shown that  $\Gal(F^{(3)}/F)$
determines the  Galois cohomology of $G_F(p) :=\Gal(F(p)/F)$ with $\mathbb{F}_p$ coefficients!
In \cite{idojan3},  an even smaller Galois group over $F$; namely $\Gal(F^{\{3\}}/F)$, was shown to have  this property.
Therefore, Theorem~A answers the natural question: ``For which fields $F$, does one have $F^{\{3\}}=F^{(3)}$?'',
a question which has its own importance beyond the connection to $p$-rigid fields.

Furthermore, we provide a Galois-theoretic characterization for $p$-rigid fields, together with
an explicit description of the maximal $p$-extension $F(p)$ of a $p$-rigid field $F$,
and of its maximal pro-$p$ Galois group $G_F(p)$.
In particular, we prove the equivalence of the following three statements:
$G_F(p)$ is solvable; $F$ is $p$-rigid; and $F$ is hereditarily $p$-rigid
(see Theorem~\ref{thm:rigidheredrigid} and Corollary~\ref{corF}).
Although this result is proved in \cite[Prop.~2.2]{englerkoen},
this earlier proof is less direct than our approach. In particular, it relies on a number
of results on Henselian valuations, covered in several papers, and on
some results of \cite{ware}. On the other hand, we refer to \cite{ware} only for some definitions, and we develop and prove our results
independently of both \cite{ware} and \cite{englerkoen}.
In fact our approach is substantially different from the approach of Engler and Koenigsmann,
as our proofs use only elementary methods from Galois theory and the theory of cyclic algebras -- in the spirit of Ware's paper.

Using relatively simple argument but powerful Serre's theorem on cohomological dimensions of open subgroups of pro-$p$ groups and corollary of  Rost-Voevodsky's proof of the Bloch-Kato conjecture we are able to prove ``going down $p$-rigidity theorem" in the case when $G_F(p)$ is finitely generated; see Theorem  \ref{goingdown}. Then, using this result and the well-known Lazard's group theoretic characterization of $p$-adic
analytic pro-$p$ groups we are able to show that if $G_F(p)$ is finitely generated then $F$ is $p$-rigid if and only if $G_F(p)$ is a $p$-adic
analytic pro-$p$ group.

We also investigate how $p$-rigid fields are related to certain powerful pro-$p$ groups,
studied by the third author in \cite{claudio}, and with certain Galois modules, studied by J. Swallow and the second
author in \cite{janswallow}.

Maximal pro-$p$ Galois groups play a fundamental role in the study of absolute Galois groups of fields.
Moreover, the cases where $F$ is a $p$-rigid field and where $G_F(p)$ is a free pro-$p$ group,
or a Demu\v{s}kin group, are cornerstones in the study of maximal pro-$p$ Galois groups.
Therefore, it is important to have a clear, complete and explicit description of the former case.
In fact we will be able to recover the entire group $G_F(p)$ from rather small Galois groups and the structure
of the $p$th roots of unity contained in $F$.

This paper is organized as follows.
In Section~\ref{sec:2prel} we review some preliminary definitions and basic facts
about pro-$p$ groups and their cohomology.
In Section~\ref{sec:3rigfiel} we state and prove results on the cohomology
and the group structure of the Galois group $G_F(p)$,
which will be used in proving that $p$-rigidity implies hereditary $p$-rigidity,
and in characterizing $G_F(p)$ for $p$-rigid fields.
Finally in Section~\ref{sec:4proof} we prove Theorem~A and we study the connections with the
fast-solvability of equations and other group-theoretic consequences of $p$-rigidity.

\section{Preliminaries}\label{sec:2prel}

\subsection{Pro-$p$ groups}\label{subsec:2.1pgps}
Henceforth we will work in the category of pro-$p$ groups and assume that all our subgroups of pro-$p$ groups will be closed.
Let $G$ be a pro-$p$-group. For $\sigma,\tau$ in $G$, $^\sigma{\tau}:=\sigma\tau\sigma^{-1}$,
and $[\sigma,\tau]:= {^\sigma{\tau}}\cdot\tau^{-1}$ is the commutator of $\sigma$ and $\tau$.
 The closed subgroup of $G$ generated by all of the commutators, will be denoted by $[G,G]$.
Note that \cite{analytic} has a slighlt different convention for commutators. They define as a commutator $[\sigma,\tau]$ what in our notation is $ [\sigma^{-1},\tau^{-1}]$. However, this will not effect the citation and use of results in \cite{analytic}.

For a profinite group $G$, the Frattini subgroup $\Phi(G)$ of $G$ is defined to be
the intersection of all maximal normal subgroups of $G$.
If $G$ is a pro-$p$ group, it can be shown that
\[\Phi(G):=G^p[G,G]\]
\cite[Prop.~1.13]{analytic}, where $G^p$ is the subgroup generated by the $p$-powers of the elements of $G$.
Hence $G/\Phi(G)$ is a $p$-elementary abelian group of possibly infinite rank.

We define the subgroups $\gamma_i(G)$ and $\lambda_i(G)$ of $G$ to be the elements of the lower descending central series,
resp. of the lower $p$-descending central series, of the pro-$p$ group $G$. That is, $\gamma_1(G)=\lambda_1(G)=G$, and
 \[\gamma_{i+1}(G) :=[\gamma_i(G),G],\quad \lambda_{i+1}(G) :=\lambda_i(G)^p[\lambda_i(G),G] ,\]
for $i \geq 1$.
In this terminology, note that the Frattini subgroup $\Phi(G)$ is exactly $\lambda_2(G)$.
If $G$ is finitely generated, the subgroups $\lambda_i(G)$ make up a system of open neighborhoods of 1 in $G$.

Finally,  we denote by $d(G)$ the minimal number of generators of $G$. It follows that $d(G)=\dim(G/\Phi(G))$ as an $\F_p$-vector space.
If $d=d(G)$, we say that $G$ is $d$-generated. If $G$ is a finitely generated, then
the rank $\rank(G)$ of the group $G$ is
\[\rank(G)=\sup_{C\leq G}\{d(C)\}=\sup_{C\leq G}\{d(C)|C\text{ is open}\}\]
(see \cite[\S 3.2]{analytic}).

\subsection{Maximal pro-$p$ Galois groups and their cohomology}\label{subsec:2.2cohom}
Consider $\F_p$ as trivial $G$-module.
The cohomology groups $H^k(G,\F_p)$ of $G$ with coefficients in $\F_p$ are defined for
all $k\geq0$.
In particular,
\begin{equation}\label{lowercohomology}
 H^0(G,\F_p)=\F_p\quad\text{and}\quad H^1(G,\F_p)=\Hom(G,\F_p).
\end{equation}
By Pontryagin duality it follows that
\begin{equation}\label{duality}
 H^1(G,\F_p)=G^\vee=\left(G/\Phi(G)\right)^\vee\quad\text{and}\quad d(G)=\dim_{\F_p}\left(H^1(G,\F_p)\right),
\end{equation}
where the symbol $\_^\vee$ denotes the Pontryagin dual (see \cite[Ch. III \S9]{nsw}).
The {\it cohomological dimension} $\cohdim(G)$ of a (pro-)$p$ group $G$ is the least positive integer
$k$ such that $H^{k+1}(G,\F_p)=0$, and if such $k$ does not exist, one sets $\cohdim(G)=\infty$.

The direct sum $H^\bullet(G,\F_p)=\bigoplus_{k\geq0}H^k(G,\F_p)$, is equipped with the {\it cup product}
\begin{displaymath}
 \xymatrix{H^r(G,\F_p)\times H^s(G,\F_p)\ar[r]^-{\cup} & H^{r+s}(G,\F_p),}
\end{displaymath}
which gives it a structure of a graded commutative $\F_p$-algebra.
For further facts on the cohomology of profinite groups we refer the reader to \cite{nsw}.

We say that a pro-$p$ group $G$ is a {\it Bloch-Kato} pro-$p$ group
if for every closed subgroups $C$ of $G$ the $\F_p$-cohomology algebra $H^\bullet(C,\F_p)$
is {\it quadratic}, i.e., it is generated by $H^1(C,\F_p)$ and the relations are generated
as ideal by elements in $H^2(C,\F_p)$ (see \cite{claudio}).

Given a field $F$, let $\bar{F}^s$ denote the separable closure of $F$,
and let $F(p)$ be the maximal $p$-extension of $F$, i.e., $F(p)$ is the compositum of all finite Galois extensions
$K/F$ of $p$-power degree.
Then $G_F:=\Gal(\bar{F}^s/F)$ is the {\it absolute Galois group} of $F$, and
the {\it maximal pro-$p$ Galois group} $G_F(p)$ of $F$ is the maximal pro-$p$ quotient of $G_F$ or, equivalently,
$G_F(p)$ is the Galois group of the maximal $p$-extension $F(p)/F$.
We then have the Galois correspondence, according to which the closed subgroups of $G_F(p)$ correspond to subextensions of $F(p)/F$ and conversely.

By the proof of the Bloch-Kato conjecture, obtained by M. Rost and V, Voevodsky
(with C. Weibel's patch), one knows that the maximal pro-$p$ Galois group of a
field containing the $p$th roots of unity is a Bloch-Kato pro-$p$ group \cite{weibel, weibel2, voev}.

Another important feature of maximal pro-$p$ Galois groups is the following:
if $p$ is odd then $G_F(p)$ is torsion-free. This Artin-Schreier type result is due to E. Becker (see \cite{beckerAS}).

The study of maximal pro-$p$ Galois groups is extremely important, since they are easier to handle
than absolute Galois groups; yet they provide substantial information about absolute Galois groups
and the structure of their base fields.

\subsection{Important subextensions}\label{subsec:2.3F3}
Throughout this paper, fields are assumed to contain the $p$th roots of unity.

Given a field $F$, let $F^{(2)}=F(\sqrt[p]{F})$, i.e., $F^{(2)}$ is the compositum of all the extensions
\[F\left(\sqrt[p]{a}\right),\quad\text{with }a\in \dot{F}.\]
For $n\geq3$, we define recursively the extensions $F^{\{n\}}/F$ and $F^{(n)}/F$ in the following way:
\begin{itemize}
 \item the field $F^{\{n\}}$ denotes the compositum of all the extensions
\[F^{\{n-1\}}\left(\sqrt[p]{\gamma}\right),\quad\text{with }\gamma\in\dot{F^{\{n-1\}}}\smallsetminus\left(\dot{F^{\{n-1\}}}\right)^p,\]
where we put $F^{(2)}$ instead of $F^{\{2\}}$
 \item the field $F^{(n)}$ denotes the compositum of all the extensions
\[F^{(n-1)}\left(\sqrt[p]{\gamma}\right),\quad\text{with }\gamma\in \dot{F^{(n-1)}}\smallsetminus\left(\dot{F^{(n-1)}} \right)^p\]
such that $F^{(n-1)}(\sqrt[p]{\gamma})/F$ is Galois.
\end{itemize}
Notice that all extensions $F^{\{n\}}/F$ and $F^{(n)}/F$ are Galois.

\begin{prop}\label{prop:maxpextwithourfields}
 For any field $F$ one has
\[F(p)=\bigcup_{n>1}F^{(n)}.\]
\end{prop}

\begin{proof}
 The inclusion $F(p)\supseteq\bigcup_n F^{(n)}$ is obvious.

For the converse, let $K/F$ be a finite Galois $p$-extension of degree $|K:F|=p^m$, for some $m\in\N$.
Then, by the properties of finite $p$-groups, one has a chain
\[F=K_0\subset K_1\subset \ldots\subset K_m=K \]
of fields such that $K_i/F$ is Galois and $|K_i:F|=p^i$ for every $i=1,\ldots,m$.
In particular, for every $i$ the extension $K_{i+1}/K_i$ is cyclic of degree $p$,
and $\Gal(K_{i+1}/K_i)$ is central in $\Gal(K_{i+1}/F)$.

We claim that $K_i\subseteq F^{(i+1)}$ for every $i$.
This is clear for $i=0$ and $i=1$.
Assume $K_{i-1}\subseteq F^{(i)}$ by induction.
Then $K_i/K_{i-1}$ is Galois and cyclic of degree $p$, thus $K_i=K_{i-1}(\sqrt[p]{\alpha})$
with $\alpha\in F^{(i)}$.
Consequently $K_i\subseteq F^{(i+1)}$, and the statement of the Lemma follows.
\end{proof}

Let $\dot{F} :=F\smallsetminus\{0\}$.
Then, for every Galois $p$-extension $K/F$, $\dot{K}$ is a $G_F(p)$-module.
Moreover, since $\sigma.\gamma^p=(\sigma.\gamma)^p$ for any $\gamma\in\dot{F}(p)$ and $\sigma\in G_F(p)$,
one has that $\dot{K}/\dot{K}^p$ is also a $G_F(p)$-module, with $K$ as above.
We denote the module $\dot{F}^{(2)}/(\dot{F}^{(2)})^p$ by $J$.
The module $J$ has been studied in \cite{AdGaKaMi} for $p=2$, and in \cite{secret} for $p$ odd,-
where it has been shown that $J$ provides substantial information about the field $F$.

We can generalize the construction of $J$ in the following way.
For every $n\geq 3$ let
\[J_n=\frac{\dot{F^{(n)}}}{\left({\dot{F^{(n)}}}\right)^p}.\]
We set $J=J_2$.
Then each module $J_n$ is a $\F_p$-vector space and a $G_F(p)$-module.
(Note that the notation used here is different from the one in \cite{AdGaKaMi} and \cite{secret}.)

The following lemma is a well known fact from elementary Galois theory.

\begin{lem}\label{lemmaGalois}
 Let $K/F$ be a Galois $p$-extension of fields, and let $a\in \dot{K}\smallsetminus\dot{K}^p$.
Then $K(\sqrt[p]{a})/F$ is Galois if, and only if,
\[\frac{\sigma.a}{a}\in\dot{K}^p\]
for every $\sigma\in\Gal(K/F)$.
\end{lem}

It follows that
\begin{equation}\label{extensionwithJ}
F^{(n+1)}=F^{(n)}\left(\sqrt[p]{(J_n)^G}\right)
\end{equation}
for every $n\geq2$, where $J_n^G$ denotes the submodule of $J_n$ fixed by $G=G_F(p)$.

\section{Rigid fields}\label{sec:3rigfiel}
Given a field $F$, the quotient group $\dot{F}/{\dot{F}}^p$ is a $p$-elementary abelian group,
so that we may consider it as $\F_p$-vector space.
Henceforth we will always assume that $\dot{F}/\dot{F}^p$ is not trivial.
For an element $a\in\dot{F}$, $[a]_F=a{\dot{F}}^p$ denotes the coset of $\dot{F}/\dot{F}^p$ to which $a$ belongs.
In particular, $k.[a]_F=[a^k]_F$ for $k\in\F_p$, and for $a,b\in\dot{F}$, $[a]_F$ and $[b]_F$ are
$\F_p$-linearly independent if, and only if, $F(\sqrt[p]{a})\neq F(\sqrt[p]{b})$.
Moreover, let $\mu_p\subseteq F$ be the group of the roots of unity of order $p$.
Then one may fix an isomorphism $\mu_p\cong\F_p$, so that by Kummer theory one has the isomorphism
\begin{equation}\label{Kummerisomorphism}
\phi\colon\dot{F}/\dot{F}^p\tilde\longrightarrow H^1\left(G_F(p),\F_p\right),\quad
\phi\left([a]_F\right)(\sigma)=\frac{\sigma.\sqrt[p]{a}}{\sqrt[p]{a}}.
\end{equation}

\begin{defi}
Let $N$ denote the norm map $N\colon F(\sqrt[p]{a})\rightarrow F$.
An element $a\in \dot{F}\smallsetminus\dot{F}^p$ is said to be $p$-{\bf rigid} if $b\in N(F(\sqrt[p]{a}))$
implies that $b\in[a^k]_F$ for some $k\geq0$.
The field $F$ is called $p$-{\bf rigid} if every element of $\dot{F}\smallsetminus\dot{F}^p$ is $p$-rigid.
\end{defi}

In \cite{ware}, R. Ware calls a field $F$ {\it hereditarily} $p$-rigid if every $p$-extension of $F$
is a $p$-rigid field. In this paper we shall call such fields \emph{heriditary} $p$-rigid.

\begin{example}
\begin{itemize}
 \item[i.] Let $q$ be a power of a prime such that $p\mid (q-1)$.
Let $F=\F_q((X))$, namely, $F$ is the field of Laurent series on the indeterminate $X$
with coefficients in the finite field $\F_q$.
Then $F$ is $p$-rigid \cite[p. 727]{ware}.
\item[ii.] Let $\zeta$ be a primitive $p$th root of unity and $\ell$ a prime different from $p$.
Then $F=\Q_\ell(\zeta)$ is $p$-rigid.
Indeed, by \cite[Prop.~7.5.9]{nsw} the maximal pro-$p$ Galois group $G_F(p)$
is 2-generated and it has cohomological dimension $cd(G_F(p))=2$.
Hence $G_F(p)$ satisfies Corollary~\ref{corE}, and $F$ is $p$-rigid.
\end{itemize}
\end{example}

\subsection{Powerful pro-$p$ groups}\label{subsec:3.1pwflgps}
A pro-$p$ group $G$ is said to be {\it powerful} if
\[[G,G]\subseteq\left\{\begin{array}{cc}G^p & \text{for }p\text{ odd},\\G^4 & \text{for }p=2,\end{array}\right.\]
where $[G,G]$ is the closed subgroup of $G$ generated by the commutators of $G$,
and $G^p$ is the closed subgroup of $G$ generated by the $p$-powers of the elements of $G$.

Moreover, a finitely generated pro-$p$ group $G$ is called {\it uniformly powerful},
or simply {\it uniform}, if $G$ is powerful, and
\[\left|\lambda_i(G):\lambda_{i+1}(G)\right|=|G:\Phi(G)|\quad\text{for all }i\geq1.\]
A finitely generated powerful group is uniform if and only if it is torsion-free (see \cite[Thm. 4.5]{analytic}).
Finally, a pro-$p$ group $G$ is called {\it locally powerful} if every finitely generated closed subgroup
of $G$ is powerful.

In order to state the classification of torsion-free, finitely generated, locally powerful pro-$p$ groups
-- which we shall use to describe explicitly the maximal pro-$p$ groups of rigid fields in \S~\ref{subsec:3.2G_Fforrigid}
-- we shall introduce the notion of {\it oriented} pro-$p$ groups.

\begin{defi}
 A pro-$p$ group $G$ together with a (continuous) homomorphism $\theta\colon G\rightarrow\Z_p^\times$
is called an {\bf oriented} pro-$p$ group, and $\theta$ is called the orientation of $G$.
If one has that $ghg^{-1}=h^{\theta(g)}$ for every $h\in\kernel(\theta)$ and every $g\in G$,
then $G$ is said to be {\bf$\theta$-abelian}.
\end{defi}

The above definition generalizes to all pro-$p$ groups the notion of cyclotomic character
of an absolute (and maximal pro-$p$) Galois group.
Such homomorphism has been studied previously for maximal pro-$p$ Galois groups in \cite{small}
(where it is called ``cyclotomic pair''), and in \cite{koenig3} for absolute Galois groups.
(See also \cite{jacobeare} for the case $p=2$.)

\begin{prop}[\cite{claudio}, Proposition 3.4]\label{prop:presentationthetabelian}
 Let $G$ be an oriented pro-$p$ group with orientation $\theta$.
Then $G$ is $\theta$-abelian if and only if there exists a minimal set of generators $\{x_\circ,x_i|i\in\mathcal{I}\}$
for some set of indices $\mathcal{I}$, such that $G$ has a presentation
\begin{equation}\label{eq:presentationthetabelian}
G=\left\langle x_\circ,x_i\left|[x_\circ,x_i]=x_i^{\theta(x_\circ)-1},[x_i,x_j]=1,i,j\in\mathcal{I}\right.\right\rangle.
\end{equation}
I.e., $G\cong\Z_p\ltimes Z$, with $Z\cong\Z_p^{\mathcal{I}}$,
and the action of the first factor on $Z$ is the multiplication by $\theta(x_\circ)$.
\end{prop}

\begin{rem}
 Notice that the statement of \cite[Prop.~3.4]{claudio} refers only to finitely generated pro-$p$ groups,
yet the proof does not use this fact, so that it holds also for infinitely generated pro-$p$ groups.
\end{rem}

In fact, torsion-free, finitely generated, locally powerful pro-$p$ groups and $\theta$-abelian groups coincide.

\begin{thm}[\cite{claudio}, Thm. A] \label{thm:thetabelianifflocally}
 A finitely generated uniform pro-$p$ group $G$ is locally powerful
if and only if  there exists an orientation $\theta\colon G\rightarrow\Z_p^\times$
such that $G$ is $\theta$-abelian.
\end{thm}

Actually, it is possible to extend the above result to infinitely generated pro-$p$ groups.

\begin{prop}\label{prop:thetabelianifflocally-infinite}
A locally powerful torsion-free pro-$p$ group $G$ is $\theta$-abelian
for some orientation $\theta\colon G\rightarrow\Z_p^\times$.
\end{prop}

\begin{proof}
By Theorem~\ref{thm:thetabelianifflocally}, we are left to the case when $G$ is infinitely generated.
If $G$ is abelian, then $G$ is $\theta$-abelian with $\theta\equiv{\mathbf 1}$.
Hence, suppose $G$ is non-abelian.




Let $C<G$ be any finitely generated subgroup.
Thus $C$ is $\theta_C$-abelian, for some homomorphism $\theta_C$.
In particular, let $H_C=[C,C]$ be the commutator subgroup of $C$, and let $Z_C=\kernel(\theta_C)$.
Then $Z=\Cent_C(H)$, and $H_C=Z_C^{\lambda_C}$, for some $\lambda_C\in p\Z_p$.

Let $H=[G,G]$ be the commutator subgroup of $G$, and let $Z\leq G$ be the subgroup generated
by all the elements $y\in G$ such that $y^\lambda\in H$ for some $\lambda\in p\Z_p$.
Then one has
\[H=\overline{\bigcup_{C<G}H_C}\quad\text{and}\quad Z=\overline{\bigcup_{C<G}Z_C},\]
where $\overline{*}$ denotes the pro-$p$ closure inside $G$.
Notice that all the $H_C$ and the $Z_C$ (and thus also $H$ and $Z$) are abelian.
In particular, $G\supsetneq Z$, since $G$ is non-abelian, and $G/Z\cong \Z_p$.

For every element $x\in G\smallsetminus Z$, one has $[x,Z]=Z^{\lambda_x}$ for some $\lambda_x\in p\Z_p$
and take $x_0$ among all such $x$ such that $\lambda_{x_0}$ is minimal $p$-adic value.
Define the homomorphism $\theta\colon G\rightarrow \Z_p^\times$
such that $\kernel(\theta)=Z$ and $\theta(x_0)=1+\lambda_{x_0}$.
Then $\theta|_C=\theta_C$ for every finitely generated large enough subgroup $C<G$, so that $G$ is in fact $\theta$-abelian.
\end{proof}

\begin{rem}
\begin{itemize}
 \item[i.] Notice that, although the theory of powerful pro-$p$ groups works effectively only for finitely generated groups,
it extends nicely to the infinitely generated case when we assume local powerfulness.
\item[ii.] It is possible to prove Proposition~\ref{prop:thetabelianifflocally-infinite} using methods from
Lie theory, since every uniformly powerful pro-$p$ group $G$ is associated to a $\Z_p$-Lie algebra $\log(G)$
(see \cite[\S4.5]{analytic} and \cite[\S3.1]{claudio}).
In the case of a locally powerful group, such $\Z_p$-Lie algebra has a very simple shape,
so that it is possible to ``linearize'' the proof.
\end{itemize}
\end{rem}

\subsection{The maximal pro-$p$ Galois group of a rigid field}\label{subsec:3.2G_Fforrigid}
Throughout this subsection we shall denote the maximal pro-$p$ Galois group $G_F(p)$ simply by $G$.
Let $a,b\in\dot{F}$.
The cyclic algebra $(a,b)_F$ is the $F$-algebra generated by elements $u,v$ subject to the relations
$u^p=a$, $v^p=b$ and $uv=\zeta_pvu$, where $\zeta_p$ is a $p$th primitive root of unity.

From \cite[Ch. XIV, \S 2, Proposition 5]{SerreLF}, one knows that $[(a,b)_F]=1$ in the Brauer group $\bra(F)$ if and only if
$\chi_a\cup\chi_b=0$ in $H^2(G,\F_p)$, with $\chi_a=\phi([a]_F)$ and $\chi_b=\phi([b]_F)$ as in (\ref{Kummerisomorphism}).
(For the definition and the properties of the Brauer group of a field see \cite[Ch. 2]{GilleSzamuely}.)
Moreover, it is well known that $[(a,b)_F]=1$ if, and only if, $b$ is a norm of $F(\sqrt[p]{a})$.
Therefore, $F$ is $p$-rigid if, and only if, the map
\begin{equation}\label{wedgecup}
 \Lambda_2(\cup)\colon H^1(G,\F_p)\wedge H^1(G,\F_p)\longrightarrow H^2(G,\F_p),
\end{equation}
induced by the cup product, is injective.

The following theorem is due to P. Symonds and Th. Weigel.
\begin{thm}[\cite{symondsthomas}, Thm. 5.1.6]\label{thmsymondsthomas}
Let $G$ be a finitely generated pro-$p$ group.
Then the map
\[\Lambda_2(\cup)\colon H^1(G,\F_p)\wedge H^1(G,\F_p)\longrightarrow H^2(G,\F_p)\]
is injective if, and only if, $G$ is powerful.
\end{thm}
By (\ref{duality}), (\ref{Kummerisomorphism}), and \cite{janswallow} this implies the following.

\begin{prop}\label{prop:rigidiffpowerful}
Assume that $\dim_{\F_p}(\dot{F}/\dot{F}^p)$ is finite.
Then $F$ is rigid if and only if $G$ is powerful.
Moreover, $F$ is hereditary $p$-rigid if and only if $G$ is locally powerful.
\end{prop}

\begin{rem}\label{ref:rigidiffpwflinfty}
The hereditary $p$-rigidity of $F$ implies the locally
powerfulness of $G$ for all fields by Proposition \ref{prop:rigidiffpowerful} and the definition of locally powerful groups.
\end{rem}

Furthermore, it is possible to deduce that in the case $\dot{F}/\dot{F}^p$ is finite,
$p$-rigidity implies hereditary $p$-rigidity.
Indeed, for a field $F$ such that $\dim_{\F_p}(\dot{F}/\dot{F}^p)<\infty$, by \cite[Thm.~B]{claudio}
one has that either $G_F(p)$ is locally powerful, or it contains a closed non-abelian free pro-$p$ group.
For a $p$-rigid field, $G_F(p)$ is powerful, and therefore necessarily it is locally powerful,
since powerful pro-$p$ groups contain no closed non-abelian free pro-$p$ groups,
and thus $F$ is hereditary $p$-rigid.

\begin{rem}
 Ware provided the same description for the maximal pro-$p$ Galois group $G_F(p)$ of a hereditary $p$-rigid field $F$,
but with the further assumption that $F$ contains also a primitive $p^2$th root of unity \cite[Thm.~2]{ware}.
The third author already got rid of such assumption in \cite[Cor.~4.9]{claudio}.
\end{rem}

\begin{rem}
Observe that one can prove directly  for all fields $F$ that if $G = G_F(p)$ is poweful then $F$ is $p$-rigid.
Indeed if we assume using a contradiction argument that $G$ is powerful but $F$ is not $p$-rigid, then by \cite[Lemma 4]{ware1}, $G$ will have as a quotient the group $H_{p^3}$ (the unique non-abelian group of order $p^3$ and exponent $p$.) But this means that $G/G^p$ is non-abelian and therefore $G$ is not powerful. On the other hand, from the explicit form of $G_F(p)$ for each rigid field, we shall see (Corollary \ref{corB} and Theorem  \ref{prigid}) that $G/G^p$ is abelian and hence $G$ is powerful. Thus $F$ is $p$-rigid if and only if $G$ is powerful.
\end{rem}

\subsection{Rigidity implies hereditary rigidity}\label{subsec:3.3rigidityheredrigidity}
As above, let $G=G_F(p)$.
Let $\bra_p(F)$ denote the subgroup of $\bra(F)$ consisting of elements of order $p$.
From Merkuryev and Suslin's work, an element of $\bra_p(F)$ is a product of cyclic algebras,
i.e., one has the following commutative diagram:
\begin{displaymath}
 \xymatrix{ \dot{F}/\dot{F}^p\wedge\dot{F}/\dot{F}^p\ar[d]^{\phi\wedge\phi}\ar@{->>}[rr] && \bra_p(F)\ar[d]\\
H^1(G,\F_p)\wedge H^1(G,\F_p)\ar@{->>}[rr]^-{\Lambda_2(\cup)} && H^2(G,\F_p) }
\end{displaymath}
where the vertical arrows are isomorphisms, and $\phi$ is the Kummer isomorphism as in (\ref{Kummerisomorphism}).
Therefore $\bra_p(F)$ is a quotient of $\dot{F}/\dot{F}^p\wedge\dot{F}/\dot{F}^p$.
(In particular, if $F$ is $p$-rigid, also the horizontal arrows are isomorphisms.)
Hence, the following hold in $\bra_p(F):$
\begin{eqnarray}\label{cyclicalgebras0}
 && \left[(b,a)_F\right]=\left[(a,b)_F\right]^{-1}\\
 && \left[(ab,c)_F\right]=\left[(a,c)_F\right]\cdot\left[(b,c)_F\right] \label{cyclicalgebras1}\\
 && \left[(a^k,b)_F\right]=\left[(a,b)_F\right]^k=\left[(a,b^k)_F\right]\label{cyclicalgebras2}
\end{eqnarray}
for every $a,b,c\in\dot{F}$ and $k\in\F_p$.

Let $E/F$ be a cyclic extension of degree $p$, namely, $E=F(\sqrt[p]{a})$ with $a\in\dot{F}\smallsetminus\dot{F}^p$.
Let
\[\epsilon\colon\dot{F}/\dot{F}^p\longrightarrow\dot{E}/\dot{E}^p\]
be the homomorphism induced by the inclusion $F\hookrightarrow E$.

\begin{lem}\label{lem:janswallow}
Let $F$ be $p$-rigid. For $E/F$ as above, one has
\[\dot{E}/\dot{E}^p=\left\langle[\sqrt[p]{a}]_F\right\rangle\oplus\epsilon\left(\dot{F}/\dot{F}^p\right)\]
as $\F_p$-vector space.
\end{lem}

\begin{proof}
By \cite[Thm. 1]{janswallow}, one has that $J=X\oplus Z$ as $G$-module,
where $X$ is an irreducible $G$-module, and $Z$ is a trivial $G$-module.
Obviously, one has the following inclusions:
\begin{equation}\label{lemmajanswall}
 \left\langle[\sqrt[p]{a}]_F\right\rangle\subseteq X,\quad \epsilon\left(\dot{F}/\dot{F}^p\right)\subseteq J^G,\quad
Z\subseteq J^G.
\end{equation}
Fix a primitive $p$th root $\zeta_p$.
If $\zeta_p$ is a norm of $E/F$, then by \cite[Cor.~1]{janswallow} one has $\dim(X)=1$,
so that $J^G=J$, and by \cite[Lemma 2]{janswallow} the claim follows.
Otherwise, by \cite[Cor.~1]{janswallow} one has $\dim(X)=2$,
so that necessarily
\[X=\left\langle[\sqrt[p]{a}]_F,[\zeta_p]_F\right\rangle,\]
for $[\zeta_p]_F=(\sigma-1)[\sqrt[p]{a}]_F$, with $\sigma$ a suitable generator of $\Gal(E/F)$,
and by \cite[Lemma 2]{janswallow} one has $J^G=\epsilon(\dot{F}/\dot{F}^p)$,
so that by (\ref{lemmajanswall}) the claim follows.
\end{proof}

\begin{lem}\label{lem:brEF}
 Let $E/F$ be as above, and let $\bra(E/F)\leq\bra(F)$ be the kernel of the morphism
\[\bra(F)\longrightarrow\bra(E),\quad \left[(b,c)_F\right]\longmapsto\left[(b,c)_F\otimes_FE\right].\]
Then $[(b,c)_F]\in\bra(E/F)$ if, and only if, $[(b,c)_F]=[(a,d)_F]$ for some $d\in\dot{F}$.
\end{lem}

\begin{proof}
Let $\bar{G}\cong\Z/p.\Z$ be the Galois group of $E/F$, and fix a primitive $p$th root $\zeta_p$.
It is well known that
\begin{equation}\label{braueriso}
\bra(E/F)\cong H^2(\bar{G},\dot{E})\quad\text{and}\quad
H^2(\bar{G},\dot{E})\cong \dot{F}/N_{E/F}(\dot{E}).
\end{equation}
Namely, by the first isomorphism of (\ref{braueriso}), every element $[A]$ of $\bra(E/F)$ can be represented by
a cross-product $F$-algebra $A$ induced by a cocycle $z\colon\bar{G}\times\bar{G}\rightarrow\dot{E}$,
and by the second isomorphism of (\ref{braueriso}), the image of $z$ is $\{1,d\}$ with $d\in\dot{F}\smallsetminus N_{E/F}(\dot{E})$,
and $A$ has a presentation such that $A=(a,d)_F$.
\end{proof}

\begin{thm}\label{thm:rigidheredrigidEF}
Let $E=F(\sqrt[p]{a})$ with $a\in\dot{F}\smallsetminus\dot{F}^p$. If $F$ is $p$-rigid, then so
is $E$.
\end{thm}

\begin{proof}
In order to prove that $E$ is $p$-rigid, we have to show that for $\alpha,\beta\in\dot{E}$,
one has $[(\alpha,\beta)_E]=1$ in $\bra(E)$ if, and only if, $[\alpha]_E,[\beta]_E$ are
$\F_p$-linearly dependent in $\dot{E}/\dot{E}^p$.

Thus, suppose for contradiction that $[\alpha]_E,[\beta]_E$ are $\F_p$-linearly independent but $[(\alpha,\beta)_E]=1$.
By Lemma \ref{lem:janswallow}, and by (\ref{cyclicalgebras0}), (\ref{cyclicalgebras1}) and (\ref{cyclicalgebras2}), we can reduce without loss
of generality to the following to cases:
either $\alpha,\beta\in\dot{F}$, or $\alpha=\sqrt[p]{a}$ and $\beta\in\dot{F}$.

\medskip

{\it 1st case:} Assume $\alpha,\beta\in\dot{F}$.
Since $[\alpha]_E,[\beta]_E$ are $\F_p$-linearly independent, so are $[\alpha]_F,[\beta]_F$ in $\dot{F}/\dot{F}^p$.
Thus, by $p$-rigidity of $F$, $[(\alpha,\beta)_F]\neq1$ in $\bra(F)$.
Since we are assuming that
\[\left[(\alpha,\beta)_E\right]=\left[(\alpha,\beta)_F\otimes_FE\right]=1\quad\text{in }\bra(E),\]
it follows that $[(\alpha,\beta)_F]\in\bra(E/F)$.
Therefore, by Lemma \ref{lem:brEF}, there exists $b\in\dot{F}$ such that $[(\alpha,\beta)_F]=[(a,b)_F]$.
Since $\bra_p(F)\cong\dot{F}/\dot{F}^p\wedge\dot{F}/\dot{F}^p$, it follows that
\begin{eqnarray*}
&& [\alpha]_F\wedge[\beta]_F=[a]_F\wedge[b]_F \quad \text{in }\dot{F}/\dot{F}^p\wedge\dot{F}/\dot{F}^p\\
&\text{thus}& [\alpha]_E\wedge[\beta]_E=[a]_F\wedge[b]_E \quad \text{in }\dot{E}/\dot{E}^p\wedge\dot{E}/\dot{E}^p
\end{eqnarray*}
so that $[\alpha]_E$ and $[\beta]_E$ are not linearly independent, as $[a]_E=1$,
a contradiction.

\medskip

{\it 2nd case:} Assume $\alpha=\sqrt[p]{a}$ and $\beta\in\dot{F}$.
Let $C$ be the maximal pro-$p$ Galois group of $E$.
Then by \cite[Prop.~7.5.5]{GilleSzamuely} one has the following commutative diagrams
\begin{displaymath}
 \xymatrix{ \dot{F}/\dot{F}^p\ar[d]^{\phi_F}\ar[rr]^{\epsilon} && \dot{E}/\dot{E}^p\ar[d]^{\phi_E} &
 \dot{E}/\dot{E}^p\ar[rr]^{N_{E/F}}\ar[d]^{\phi_{E}} && \dot{F}/\dot{F}^p\ar[d]^{\phi_F} \\
H^1(G,\F_p)\ar[rr]^{\res_{G/C}^1} && H^1(C,\F_p) & H^1(C,\F_p)\ar[rr]^{\text{cor}_{G/C}^1} && H^1(G,\F_p)}
\end{displaymath}
where the vertical arrows are the Kummer isomorphisms,
together with the morphism \[\text{cor}_{E/F}\colon \bra_p(E)\longrightarrow\bra_p(F)\] induced by the
corestriction $\text{cor}_{G/C}^2\colon H^2(C,\F_p)\rightarrow H^2(G,\F_p)$.
Then by the projection formula \cite[Prop.~3.4.10]{GilleSzamuely} one has
\[\text{cor}_{G/C}^2\left(\phi_E(\sqrt[p]{a})\cup\res_{G/C}^1\left(\phi_F(b)\right)\right)
=\text{cor}_{G/C}^1\left(\phi_E(\sqrt[p]{a})\right)\cup\phi_F(b),\]
which implies
\[\text{cor}_{E/F}\left([(\sqrt[p]{a},\beta)_E]\right)=\left[(N_{E/F}(\sqrt[p]{a}),\beta)_F\right]=
\left[(a,\beta)_F\right].\]
Since $[(\sqrt[p]{a},\beta)_E]=1$ in $\bra(E)$, it follows that $[(a,\beta)_F]=1$ in $\bra(F)$.
Thus, by $p$-rigidity of $F$, $[a]_F,[\beta]_F$ are $\F_p$-linearly dependent in $\dot{F}/\dot{F}^p$,
i.e., $[\beta]_F=[a^k]_F$ for some $k\in\F_p$.
Therefore $[\beta]_E$ is trivial in $\dot{E}/\dot{E}^p$, a contradiction.
\end{proof}

The following fact is an elementary consequence of the solvability of finite $p$-groups.

\begin{fact}\label{lem:pext}
Let $K/F$ be a finite non-trivial $p$-extension with $K\subseteq F(p)$.
Then there exists a chain of extensions
\begin{equation}\label{chaiextensions}
 F=K_0\subset K_1\subset\ldots\subset K_{r-1}\subset K_r=K
\end{equation}
for some $r\geq1$, such that $|K_{i+1}:K_i|=p$ for every $i=0,\ldots,r-1$.
\end{fact}

This, together with Theorem \ref{thm:rigidheredrigidEF}, implies the following.

\begin{thm}\label{thm:rigidheredrigid}
 Every $p$-rigid field $F$ is also hereditary $p$-rigid.
\end{thm}

As mentioned in the Introduction, the above theorem was proved in a different way by Engler and Koenigsmann in
\cite[Prop.~2.2]{englerkoen}

\subsection{Galois theoretical and cohomological  characterizations for $p$-rigid fields}\label{subsec:3.4Galcharac}

In addition to Theorem \ref{thm:rigidheredrigid}, and earlier results in this section, we have the following implications.

  $F$ is $p$-rigid $\implies$  $F$ is heriditary $p$-rigid $\implies$
$G:= G_F(p)$ is locally powerful $\implies$ $G$ is $\theta$ abelian
$\implies$  $G$ has presentation as in Proposition \ref{prop:presentationthetabelian}. So in other words, we obtain the next Corollary.

\begin{cor}\label{corB}
The field $F$ is rigid if and only if there exists an orientation $\theta\colon G_F(p)\rightarrow \Z_p^\times$
such that $G_F(p)$ is $\theta$-abelian, so that $G$ has a presentation
\begin{equation}\label{eq:presentationincoro}
 G_F(p)=\left\langle\sigma,\rho_i,i\in\mathcal{I}\left|
\left[\sigma,\rho_i\right]=\rho_i^\lambda,\left[\rho_i,\rho_j\right]=1\;\forall\:i,j\in\mathcal{I}\right.\right\rangle
\end{equation}
for some set of indices $\mathcal{I}$ and $\lambda\in p.\Z_p$ such that $1+\lambda=\theta(\sigma)$.
\end{cor}

In fact if $F$ is $p$-rigid, then the suitable orientation for $G_F(p)$ is the cyclotomic character
of $F$ -- as one would expect, and as we shall see this again explicitly in Section ~\ref{4.2}.

Also, in Section \ref{4.2}, we will obtain together with the results of previous section, self-contained field theoretic proof of this corollary. (Alternatively, one can also deduce this corollary using valuation theory in \cite{englerkoen} .)

\begin{defi}
We say that pro-$p$ group $G$ is solvable if it admits a normal series of closed subgroups such that
each successive quotient is abelian. That is, we have a sequence of closed subgroups
\[1 = G_0 \le G_1 \le G_2 \le \cdots \le G_{k-1} \le G_k = G\]
such that each $G_j$ is closed and normal in $G_{j+1}$ and $G_{j+1}/G_j$ is abelian for all $j$.
\end{defi}

\begin{cor}\label{corC}
The field $F$ is $p$-rigid if and only if the maximal pro-$p$ Galois group $G_F(p)$ is solvable.
\end{cor}

\begin{proof}
 If $F$ is $p$-rigid, then by Corollary~\ref{corB} $G_F(p)$ has a presentation as in (\ref{eq:presentationincoro}),
so that $G_F(p)$ is meta-abelian (i.e., its commutator is abelian), and thus solvable; in fact, the desired normal series is
\[\ 1 = [[G, G], [G, G]] \le [G, G] \le G\]
where $G= G_F(p)$.
Conversely, if $G_F(p)$ is solvable, than it contains no closed non-abelian free pro-$p$ subgroups.
Hence, by \cite[Thm.~B]{claudio}, $G_F(p)$ is $\theta$-abelian for some orientation $\theta$,
and by Corollary~\ref{corB}, $F$ is $p$-rigid.
\end{proof}

As mentioned in the Introduction, Corollary~\ref{corB} can be deduced also using valuations techniques,
as in \cite[\S~1]{englerkoen} and \cite[Ex. 22.1.6]{Ktheoryefrat},
whereas Corollary~\ref{corC} is the double implication (ii)$\Leftrightarrow$(vi)
in \cite[Prop.~2.2]{englerkoen}.

\begin{cor}\label{corD}
The field $F$ is $p$-rigid if and only if
\begin{equation}\label{cohomologyexterior}
 H^\bullet\left(G_F(p),\F_p\right)\cong\bigwedge_{k=1}^{d(G_F(p))}\left(H^1(G_F(p),\F_p)\right).
\end{equation}
\end{cor}

\begin{proof}
 Recall that $G_F(p)$ is a Bloch-Kato pro-$p$ group, so that the whole $\F_p$-cohomology ring $H^\bullet\left(G_F(p),\F_p\right)$
depends on $H^1(G_F(p),\F_p)$ and $H^2(G_F(p),\F_p)$. That is, $H^\bullet\left(G_F(p),\F_p\right)$ has generators in degree 1 and relations in degree 2.

If the isomorphism (\ref{cohomologyexterior}) holds then in particular the morphism (\ref{wedgecup}) is injective,
and $F$ is $p$-rigid.
Conversely, if $F$ is $p$-rigid then the morphism (\ref{wedgecup}) is an isomorphism, and
\begin{displaymath}
 \xymatrix{H^1(G_F(p),\F_p)\wedge H^1(G_F(p),\F_p)\ar[r]^-{\sim} & H^2(G_F(p),\F_p).}
\end{displaymath}
Therefore the whole $\F_p$-cohomology ring is isomorphic to the exterior algebra
generated by $H^1(G_F(p),\F_p)$.
\end{proof}

Ware proved the same result in the case $\dim(\dot{F}/\dot{F}^p)<\infty$, but with the further assumptions that
$F$ is hereditary $p$-rigid and that it contains a primitive $p^2$th root of unity \cite[Thm.~4 and Corollary]{ware}.
Moreover, his proof requires computations involving the Hochschild-Serre spectral sequence.

\begin{rem}
 Clearly Corollaries~\ref{corB} and \ref{corD} hold also for every $p$-extension $K/F$.
\end{rem}

\begin{cor}\label{corE}
Given a field $F$, assume that $\dim(\dot{F}/\dot{F}^p)=d<\infty$.
Then $F$ is $p$-rigid if and only if $\cohdim(G_F(p))=d$.
\end{cor}

\begin{proof}
 If $F$ is $p$-rigid, then $\cohdim(G_F(p))=d(G_F(p))$ by Corollary~\ref{corD}, and $d(G_F(p))=d$ by $(\ref{duality})$.

Conversely, if $\cohdim(G_F(p))=d(G)=d$, then one has the isomorphism (\ref{cohomologyexterior}),
since a non-trivial relation in $H^1(G_F(p),\F_p)\wedge H^1(G_F(p),\F_p)$ would imply that
\[H^d\left(G_F(p),\F_p\right)=\spa_{\F_p}\{\chi_1\cup\cdots\cup\chi_d\}=0,\]
with $\{\chi_1,\ldots,\chi_d\}$ a basis for $H^1(G_F(p),\F_p)$, a contradiction
(see \cite[Prop.~4.3]{claudio} for more details).
\end{proof}

\begin{cor}\label{corF}
Given a field $F$, assume that $\dim(\dot{F}/\dot{F}^p)=d<\infty$.
Then $F$ is $p$-rigid if, and only if, $\dim(\dot{K}/\dot{K}^p)=d$ for every finite $p$-extension $K/F$.
\end{cor}

\begin{proof}
 Suppose that $\dim(\dot{K}/\dot{K}^p)=d$ for every finite $p$-extension $K/F$.
By (\ref{duality}) and (\ref{Kummerisomorphism}), this implies that $d(C)=d$ for every open subgroup $C\leq G$.
Therefore the rank of $G$ is finite, and $G$ contains no closed non-abelian free pro-$p$ subgroups.
Thus, by \cite[Thm.~B]{claudio}, $G_F(p)$ is powerful, and $F$ is $p$-rigid; also see \cite{idojan3}.

Conversely, if $F$ is $p$-rigid, then $G_F(p)$ is uniformly powerful (and finitely generated by hypothesis),
and by \cite[Prop.~4.4]{analytic} one has $d(C)=d(G_F(p))$ for every open subgroup $C\leq G_F(p)$, i.e.,
$\dim(\dot{K}/\dot{K}^p)=d$ for every finite $p$-extension $K/F$.
\end{proof}

\begin{rem}
 A pro-$p$ group $G$ has {\it constant generating number on open subgroups} if
\begin{equation}\label{iwasawa}
 d(C)=d(G)\quad\text{for all open subgroups }C\leq G.
\end{equation}
By Corollary~\ref{corF}, a maximal pro-$p$ Galois group has property (\ref{iwasawa}) if, and only if, $F$ is $p$-rigid.
The problem to classify all profinite groups with property (\ref{iwasawa}) was raised by K. Iwasawa (see \cite[\S 1]{klosno}).
Thus, Corollary~\ref{corF} classifies all such groups in the category of maximal pro-$p$ Galois groups (and hence also in
the category of pro-$p$ absolute Galois group).
Actually, \cite[Thms. A and B]{claudio} gives implicitly the same classification for the wider category of Bloch-Kato
pro-$p$ groups.

A similar classification has been proven in \cite{klosno} for the category of $p$-adic analytic pro-$p$ groups.
It is interesting to remark that the groups listed in \cite[Thm.~1.1.(1)-(2)]{klosno} have a presentation as in
(\ref{eq:presentationincoro}),
whereas the groups listed in \cite[Thm.~1.1.(3)]{klosno} cannot be realized as maximal pro-$p$ Galois groups,
for they have non-trivial torsion.
\end{rem}

\section{New characterization for $p$-rigid fields}\label{sec:4proof}
\subsection{Proof of Theorem A}\label{subsec:4.1proof}

Let $G$ be the maximal pro-$p$ Galois group $G_F(p)$, and
recall from \S \ref{subsec:2.3F3} the definition of the modules $J_n$.

Moreover, let $ G^{\{n\}}$ and $G^{(n)}$ denote the maximal pro-$p$ Galois group of $F^{\{n\}}$, resp. of $F^{(n)}$.
Then it is clear that
\begin{eqnarray*}
 && \Gal\left(F^{\{n+1\}}/F^{\{n\}}\right)=\frac{G^{\{n\}}}{\Phi\left(G^{\{n\}}\right)},\\
 &\text{and}& G^{\{n+1\}}=\Phi\left(G^{\{n\}}\right)=\left(G^{\{n\}}\right)^p\left[G^{\{n\}},G^{\{n\}}\right],
\end{eqnarray*}
whereas by (\ref{duality}) and (\ref{Kummerisomorphism}) one has
$J_n\cong H^1(G^{(n)},\F_p)=(G^{(n)})^{\vee}$, so that
\[(J_n)^G\cong H^1\left(G^{(n)},\F_p\right)^G=\left(\frac{G^{(n)}}{[G,G^{(n)}]}\right)^{\vee},\]
i.e., the $G$-invariant elements of $J_n$ are dual to the $G$-coinvariant elements of $G^{(n)}$,
which implies that
\begin{eqnarray*}
 && \Gal\left(F^{(n+1)}/F^{\{n\}}\right)=\frac{G^{(n)}}{\left(G^{(n)}\right)^p\left[G,G^{(n)}\right]},\\
 &\text{and}& G^{(n+1)}=\Phi\left(G^{(n)}\right)\left[G,G^{(n)}\right]=\left(G^{(n)}\right)^p\left[G,G^{(n)}\right]=\lambda_{n+1}(G).
\end{eqnarray*}

\begin{rem}\label{rem:J}
 \begin{itemize}
 \item[i.] The $G$-module $J$ can be defined in a purely cohomological manner
without involving the field $F$, since by Kummer theory one has the isomorphism
\[J\cong H^1(\Phi(G),\F_p)\]
as $G$-modules.
\item[ii.] Theorem~A can be stated in the following way: $F$ is rigid if, and only if,
$G^{\{n\}}=G^{(n)}$ for all $n\geq2$ or, equivalently, if, and only if, $J_2^G=J_2$.
 \end{itemize}
\end{rem}

\begin{prop}\label{prop:rigitimpliesF3}
 If $F$ is a $p$-rigid field then $F^{\{n\}}=F^{(n)}$ for all $n\geq2$.
\end{prop}

\begin{proof}
 By Corollary~\ref{corB}, the maximal pro-$p$ Galois group $G$ is a locally powerful group,
with a presentation as in (\ref{eq:presentationincoro}).
Direct computations imply that $\lambda_n(G)=G^{p^{n-1}}$ for all $n\geq2$.
In particular,
\[G^{\{3\}}=\Phi(G)^p[\Phi(G),\Phi(G)]=G^{p^2}=G^{(3)}.\]
Moreover, if we assume that $G^{\{n\}}=G^{(n)}=\lambda_n(G)$, then
\[ \lambda_{n+1}(G)\leq G^{\{n+1\}}=\lambda_n(G)^p[\lambda_n(G),\lambda_n(G)]\leq\lambda_{n+1}(G),\]
so that $G^{\{n+1\}}=\lambda_{n+1}(G)$.
Therefore, $G^{(n)}=G^{\{n\}}=\lambda_n(G)$ and, by Remark~\ref{rem:J}, $F^{\{n\}}=F^{(n)}$ for all $n\geq2$.
\end{proof}

On the other hand, if $F$ is not $p$-rigid, we have the opposite.

\begin{thm}\label{thm:nonrigidimpliesnonF3}
 If $F$ is not $p$-rigid, then $F^{\{3\}}\supsetneq F^{(3)}$.
\end{thm}

In order to prove the above theorem, we need two further lemmas.

\begin{lem}\label{lemmaKummer}
Let $E/F$ be a bicyclic extension of degree $p^2$, and let $L=F^{(2)}$, i.e.,
$E=F(\sqrt[p]{a},\sqrt[p]{b})$ with $a,b\in \dot{F}\smallsetminus\dot{F}^p$
such that $[a]_F,[b]_F$ are $\F_p$-linearly independent in $\dot{F}/\dot{F}^p$.
Assume $\gamma\in \dot{E}$.
Then $\gamma\in \dot{L}^p$ if, and only if, $\gamma\in \dot{F}\cdot\dot{E}^p$.
\end{lem}

\begin{proof}
Let $\gamma=x\delta^p$, with $x\in \dot{F}$ and $\delta\in \dot{E}$.
Then it is clear that $\gamma\in\dot{L}^p$, for $L=F(\sqrt[p]{F})$.

On the other hand, assume that $\gamma\in\dot{L}^p$.
Then, either $\gamma$ is a $p$-power in $E$, or it becomes a $p$-power via the extension $L/E$,
i.e., $\sqrt[p]{\gamma}\in L\smallsetminus E$.
Therefore, by Kummer theory, $\gamma$ is equivalent to an element $x\in \dot{F}$ modulo $\dot{F}^p$,
namely, $\gamma\in x\dot{F}^p$.
This proves the lemma.
\end{proof}

\begin{lem}\label{lemmapower}
Let $E/F$ be a cyclic extension of degree $p$, i.e., $E=F(\sqrt[p]{a})$ with $a\in \dot{F}\smallsetminus\dot{F}^p$.
Then
\[\sqrt[p]{a}\notin \dot{F}\cdot\dot{E}^p\]
\end{lem}

\begin{proof}
Let $\alpha\in \dot{F}\dot{E}^p$.
Then there exist $x\in \dot{F}$, $\gamma\in \dot{E}$ such that $\alpha=x\gamma^p$.
Thus
\[N_{E/F}(\alpha)=N_{E/F}\left(x\gamma^p\right)=x^pN_{E/F}(\gamma)^p\in\dot{F}^p,\]
with $N_{E/F}$ the norm of $E/F$.
Since $N_{E/F}(\sqrt[p]{a})=a\notin\dot{F}^p$, it follows that $\sqrt[p]{a}\notin \dot{F}\dot{E}^p$.
\end{proof}

\begin{proof}[Proof of Theorem \ref{thm:nonrigidimpliesnonF3}]
Since $F$ is not $p$-rigid, there exist two elements $a,b\in \dot{F}$ such that
$[a]_F$ and $[b]_F$ are $\F_p$-linearly independent, and $a$ is a norm of $F(\sqrt[p]{b})/F$.
Recall that the linear independence implies that $F(\sqrt[p]{a})\neq F(\sqrt[p]{b})$,
so that $E/F$ is a bicyclic extension of degree $p^2$, where $E=F(\sqrt[p]{a},\sqrt[p]{b})$.
Let $\Gal(E/F)\cong C_p\times C_p$ be generated by $\sigma,\tau$,
with $F(\sqrt[p]{a})=E^{\langle\tau\rangle}$ and $F(\sqrt[p]{b})=E^{\langle\sigma\rangle}$.
Then the lattice of the fields $L\supseteq E\supseteq F$ is the following:
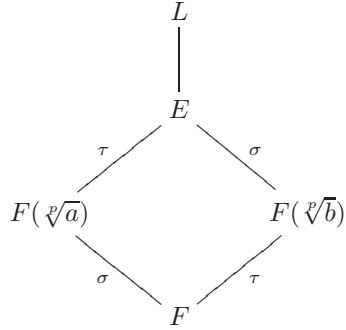
\begin{displaymath}
 \xymatrix{ & L\ar@{-}[d] & \\ & E\ar@{-}[dl]_\tau\ar@{-}[dr]^\sigma & \\
F(\sqrt[p]{a})\ar@{-}[dr]_\sigma & & F(\sqrt[p]{b})\ar@{-}[dl]^\tau \\ & F &}
\end{displaymath}

Let $\delta\in F(\sqrt[p]{b})$ such that $N_{F(\sqrt[p]{b})/F}(\delta)=a$, and let $c=\sqrt[p]{a}$.
Then
\begin{eqnarray*}
&& N_{E/F(\sqrt[p]{a})}(\delta)=\delta\cdot (\tau.\delta)\cdots(\tau^{p-1}.\delta)=N_{F(\sqrt[p]{b})/F}(\delta)=a\\
&\text{and}& N_{E/F(\sqrt[p]{a})}(c)=c\cdot(\tau.c)\cdots(\tau^{p-1}.c)=c\cdot\zeta c\cdots\zeta^{p-1}c=c^p,
\end{eqnarray*}
with $\zeta$ a primitive $p$th root of unity, so that
\[N_{E/F(\sqrt[p]{a})}\left(\frac{\delta}{c}\right)=\frac{N_{E/F(\sqrt[p]{a})}(\delta)}{N_{E/F(\sqrt[p]{a})}(c)}=\frac{a}{c^p}=1.\]
Therefore, by Hilbert's Satz 90 there exists $\gamma\in \dot{E}$ such that
\begin{equation}\label{definitiongamma}
 (\tau-1).\gamma=\frac{\tau.\gamma}{\gamma}=\frac{\delta}{c}.
\end{equation}
Suppose that $\delta/c\in \dot{L}^p$.
Then by Lemma~\ref{lemmaKummer} one has that $\delta/c\in \dot{F}\cdot\dot{E}^p$.
In particular, this implies that $c=\sqrt[p]{a}\in F(\sqrt[p]{b})^\times\cdot\dot{E}^p$,
which is impossible by Lemma~\ref{lemmapower}.
Hence $\delta/c\notin \dot{L}^p$.
Thus, by (\ref{definitiongamma}) and by Lemma~\ref{lemmaGalois} the extension $L(\sqrt[p]{\gamma})/F$
is not Galois and $\gamma^{1/p}$ is not in $F^{(3)}$.
\end{proof}

Now Proposition \ref{prop:rigitimpliesF3} and Theorem \ref{thm:nonrigidimpliesnonF3} imply Theorem A,

\begin{rem} \label{rem:prigid}
By the proof of Proposition \ref{prop:rigitimpliesF3} it follows that $F$ is $p$-rigid then $G^{(n)} = \lambda_n(G) = G^{p^{n-1}}$ for all $n >1$.
\end{rem}

\subsection{Recovering $G_F(p)$ and $F(p)$ from small Galois groups} \label{4.2}
As in Section \ref{subsec:4.1proof}, let $G$ be the maximal pro-$p$ Galois group $G_F(p)$ of the field $F$.
For $h > 0$, let $\mu_{p^h} \subseteq F(p)$  be the group of $p^h$ roots of unity. We also set $\mu_{p^\infty}$ to be the group of  all roots of unity of order $p^m$ for some $m \ge 0$. Finally we set $k \in \mathbb{N} \cup  \{ \infty \}$ to be the maximum of all $h \in \mathbb{N} \cup \{ \infty \}$ such that
$\mu_{p^h} \subseteq F$.

For a field $F$
 which is $p$-rigid, let $E/F$ be a Galois extension of degree $p$. Therefore $E = F (b^{1/p})$
 for some $b \in \dot{F} \backslash \dot{F}^p$.  Then by Kummer theory, we may choose a set of representatives $\{ b_i \colon i \in \mathcal{J}\} \subseteq \dot{F}$ of $\dot{F} \backslash  \dot{F}^p$ with $b = b_j$ for some $j \in \mathcal{J}$. Thus Lemma \ref{lem:janswallow} implies
that
\begin{equation}\label{eq:4.2E}
 \dot{E}/\dot{E}^p=\left\langle[\sqrt[p]{b_j}]_E,[b_i]_E\right\rangle_{i\neq j}.
\end{equation}
Assume now that $k<\infty$.
Then we may pick a set of representatives $\mathcal{A}=\{\zeta_{p^k},a_i,i\in\mathcal{I}\}\subseteq\dot{F}$,
where $\zeta_{p^k}$ is a fixed primitive $p^k$th root of unity,
so that $\bar{\mathcal{A}}=\{[\zeta_{p^k}]_F,[a_i]_F,i\in\mathcal{I}\}$
is a $\F_p$-basis for $\dot{F}/\dot{F}^p$.
If $k=\infty$, then we still consider a basis $\bar{\mathcal{A}}$ for $\dot{F}/\dot{F}^p$,
where the symbol $[\zeta_{p^k}]_F$ in this case is meant as an empty symbol to be ignored.
Also in this case $[\zeta_{p^{k+n}}]_E$, with $n\in\N$ and $E/F$ a $p$-extension, is also an empty symbol.
We assume that our system of roots of unity $\zeta_{p^l}$ for $l \ge 1$ in $F(p)$ is chosen such that
\[(\zeta_{p^{l+1}})^p = \zeta_{p^l}\]
for all $l \ge 1$.

Let $\mathcal{J}$ be a finite subset of $\mathcal{I}$. Set $\mathcal{J} = \{ 1, \cdots, t\}$ and let
\[ K = F(a_1^{1/p}, \cdots a_t^{1/p}, \zeta_{p^{k+1}}).\]
Then we have a series of Galois extensions
\[ F \subset K_1 \subset K_2 \cdots K_t \subset K_{t+1} = K \subseteq F^{(2)},\]
where $K_1=F(\zeta_{p^{k+1}})$, and $K_{i+1} = K_i(a_i^{1/p})$  for $1 \le i \le t$.
Then by the above arugument and induction one has
\begin{equation}
\dot{K}/\dot{K}^p = \langle [\zeta_{p^{k+1}}, [a_j^{1/p}]_K, j = 1, 2, \cdots t, [f]_K, f \in \dot{F} \rangle
\end{equation}

Assuming this observation we shall prove the following theorem.

\begin{thm}\label{lem:moduloinfinitelygeneratedcase}
 If $F$ is a $p$-rigid field, then
\[\frac{\dot{F}^{(n)}}{\left(\dot{F}^{(n)}\right)^p}=\left\langle[\zeta_{p^{k+n-1}}]_{F^{(n)}},
\left[a_i^{1/p^{n-1}}\right]_{F^{(n)}},i\in\mathcal{I}\right\rangle\]
for every $n \ge1$.
\end{thm}

\begin{proof}
Let $\bar{\mathcal{A}} = \{ [\zeta_{p^k}]_F, [a_i]_F, i \in \mathcal{I}\}$ be an $\mathcal{F}_p$ basis for
$\dot{F}/(\dot{F})^p$ as above.  We observe that for $n=1$ our statement is clear because $F^{(1)}= F$.
In order to see that our statement is also true for $n =2$ consider any  $[\alpha]_{F^{(2)}} \in \dot{F^{(2)}}/(\dot{F^{(2)}})^p $ with $\alpha$ in
$\dot{F^{(2)}}$.  Then there exists a finite subset $\mathcal{J} \subset \mathcal{I}$ such that
$\alpha \in K =  F(\zeta_{p^{k+1}}, a_j^{1/p}, j \in \mathcal{J})$. (Again, we ignore $\zeta_{p^{k+1}}$  if $k = \infty$.)
Now we see that $[\alpha]_K$ can be expressed as a product of powers of $[\zeta_{p^{k+1}}]_K, [a_j^{1/p}]_K$ and a finite number of elements $[f_l]_K$, $l=1,\cdots n, f_l \in \dot{F}$.  Passing to $F^{(2)}$, all elements $[f_j]_{[F^{(2)}]}$ become $[1]_{F^{(2)}}$.
Therefore
\begin{equation}
 \dot{F^{(2)}}/(\dot{F^{(2)}})^p=\left\langle[\zeta_{p^{k+1}}]_{F^{(2)}},
[\sqrt[p]{a_i}]_{F^{(2)}}, i \in \mathcal{I} \right\rangle
\end{equation}
This proves over statement for $n=2$. Now going from $n$ to $n+1$ is just like going from $n=1$ to $n=2$ done above taking into account that
\[ F^{(n+1)} = F^{(n)}( \zeta_{p^{k+n}}, a_i^{1/p^n} , i \in \mathcal{I}).\]
The last equality follows by induction hypothesis on $n$ and by observing that $F^{(n)} (a_i^{1/p^n}, i \in \mathcal{I})$ is Galois over $F$ for each
$i \in I$ as $\zeta_{p^{k+n-1}}$, and hence also $\zeta_{p^{n}}$ belong to $F^{(n)}$. Hence we proved our statement for all $n$.
\end{proof}

\begin{rem} \label{complicatedway}
In fact taking into account our convention about the symbol $[\zeta_{p^{k+n-1}}]_{F^{(n)}}$ when $k$ is $\infty$, one can show in a similar but slightly more complicated way as in the proof above that
\[\{[\zeta_{p^{k+n-1}}]_{F^{(n)}},
[{a_i}^{1/p^{n-1}}]_{F^{(n)}}, i \in \mathcal{I} \}\]
is a basis of $\dot{F^{(n)}}/(\dot{F^{(n)}})^p$ over $\mathbb{F}_p$ for all $n \in \mathcal{N}$.
\end{rem}

\begin{cor} \label{cor-prigid}
Assume that $F$ is a $p$-rigid field. Then we have the following.

(a) For all $n \ge 1$,
 \[ F^{(n)} = F( \zeta_{p^{k+n-1}}, a_i^{1/p^{n-1}} , i \in \mathcal{I})\]

(b)
\[ F(p) = \bigcup_{n \ge 1} F( \zeta_{p^{k+n}}, a_i^{1/p^n} , i \in \mathcal{I}).\]
\end{cor}

\begin{proof}
(a) We use Theorem \ref{lem:moduloinfinitelygeneratedcase}, its proof and induction on $n$.  The statement is true for $n=1$ because $\zeta_{p^k}$,
$a_i$, $i \in \mathcal{I}$ all belong to $F$.  Now assume that our statement is true for $n$. Using Theorem \ref{lem:moduloinfinitelygeneratedcase} and the
fact that $F^{(n)}(a_i^{1/p^n})/F$ is Galois for each $i \in \mathcal{I}$, we conclude that
 \[ F^{(n+1)} = F( \zeta_{p^{k+n}}, a_i^{1/p^{n}} , i \in \mathcal{I}).\]
This completes the induction step and we are done.

(b) This follows from the fact that
\[ F(p) = \bigcup_{n \ge 1} F^{(n)}.\]
(see Proposition \ref{prop:maxpextwithourfields}.)
\end{proof}

Now we shall determine all Galois groups $G^{[n]}:= \Gal(F^{(n)}/F)$, for all $n \ge 1$.

\begin{thm}  \label{prigid}
Suppose $F$ is a $p$-rigid field. Then we have the following.

(a)
\[ G^{[n]} =
\begin{cases}
\left( \prod_{\mathcal{I}}  \mathbb{Z}/p^{n-1} \mathbb{Z} \right) \rtimes \mathbb{Z}/p^{n-1}\mathbb{Z}  \ \ & \text{if } k < \infty\\
 \prod_{\mathcal{I}}  \mathbb{Z}/p^{n-1}\mathbb{Z} \ \ & \text{if } k = \infty
\end{cases}
\]
(b)
\[ G = \Gal(F(p)/F) =\begin{cases}
 \left( \prod_{\mathcal{I}}  \mathbb{Z}_p \right) \rtimes \mathbb{Z}_p  \ \ &\text{if } k < \infty\\
 \prod_{\mathcal{I}}  \mathbb{Z}_p \ \ &\text{if } k = \infty
\end{cases}
\]

Moreover when $k < \infty$ there exists a generator $\sigma$ of the outer factor $\mathbb{Z}/p^{n-1}\mathbb{Z}$ in (a) and of the outer factor $\mathbb{Z}_p$ in (b) such that for each $\tau$ from the inner factor $\prod_{\mathcal{I}}\mathbb{Z}/p^{n-1}\mathbb{Z}$ in (a) and each  $\tau$ from the inner factor $\prod_{\mathcal{I}} \mathbb{Z}_p$ in (b) we have
\[ \sigma \tau \sigma^{-1} = \tau^{p^k+1}.\]
\end{thm}

\begin{proof}

If $k < \infty$, consider $F^{(n)}$ as the $2$nd step extension of $F$:
\[F \subset F(\zeta_{p^{k+n-1}})  \subset F^{(n)}.\]
Then there exists $\sigma \in G^{[n]}$ such that
\[\sigma(\zeta_{p^{k+n-1}}) = \zeta_{p^{k+n-1}}^{p^{k}+1}\]
and $\sigma$ restricts to identity in $\Gal(F^{(n)}/F(\zeta_{p^{k+n-1}})$.
By standard Kummer theory (see Chapter 6, sections 8 and 9 in \cite{lang}, \cite{artintate} Chapter 6, Section 2,  and also for relevant similar calculations  in \cite{ware}, proof of Theorem 2.), we can deduce that such $\sigma$ exists and that

\begin{eqnarray*}
G^{[n]}  & = &  \Gal(F^{(n)}/F(\zeta_{p^{k+n-1}})) \rtimes \langle \sigma \rangle\\
 & \cong & \prod_{\mathcal{I}}  \mathbb{Z}/p^{n-1}\mathbb{Z} \rtimes \mathbb{Z}/p^{n-1}\mathbb{Z}\ \  \text{if } k < \infty.
\end{eqnarray*}
with the action $ \sigma \tau \sigma^{-1} = \tau^{p^k+1}$ for all $\tau \in  \Gal(F^{(n)}/F(\zeta_{p^{k+n-1}}))$.
If $k = \infty$ then direct application of Kummer theory shows that
\[ G^{[n]} = \prod_{\mathcal{I}} \Z/p^{n-1}\Z.\]

This proves (a), and (b) follows from that fact that $F(p) = \cup_{n \ge 1} F^{(n)}$. Indeed, then $G = \varprojlim G^{[n]}$, which has precisely the description in (b).

\end{proof}

If $k<\infty$, then it is well known that the Galois group of the extension $F(\mu_{p^\infty})/F$ is
pro-$p$-cyclic, i.e., $\Gal(F(\mu_{p^\infty})/F)\cong \Z_p$ \cite[Lemma 1]{ware}.
As we see from our proof of  the above theorem (part b),  the outer factor of
$G = \prod_{\mathcal{I}}  \mathbb{Z}_p  \rtimes \mathbb{Z}_p$ is isomorphic with $\Gal(F(\mu_{p^{\infty}}))$, and
$\Gal(F(p)/F(\mu_{p^\infty})) \cong \prod_{\mathcal{I}} \mathbb{Z}_p$. We can pick generators $\rho_i$, $i \in \mathcal{I}$ of
the pro $p$-group $\prod_{\mathcal{I}} \mathbb{Z}_p$ as elements of $\Gal(F(p)/F(\mu_{p^\infty}))$ such that
\[ \rho_i(a_i^{1/p^n}) = \zeta_{p^n} {a_i^{1/p^n}} \text{for all } i \in \mathcal{I}, \text{and } n \ge 1,\]
and
\[ \rho_i(a_j^{1/p^n}) = {a_j^{1/p^n}} \text{for all } j \in \mathcal{I}, j \ne i  \text{and } n \ge 1\]
 This isomorphism $\Gal(F(\mu_{p^\infty})) \cong \mathbb{Z}_p$ is induced by the cyclotomic character
\begin{equation}\label{eq:cyclotomiccharacter}
 \theta_F\colon G\longrightarrow\Aut_F(\mu_{p^\infty}),
\end{equation}
where $\Aut_F(\mu_{p^\infty})$ is the image of $\theta_F$ in $\Aut(\mu_{p^{\infty}}) \cong \Z_p^\times$.

From the above theorem we further see that we have a presentation of $G$ by generators and relations as follows:
\begin{equation}\label{eq:presentationwithcyclotomic}
  G=\left\langle\sigma,\rho_i,i\in\mathcal{I}\left|
\left[\sigma,\rho_i\right]=\rho_i^{p^k},\left[\rho_i,\rho_j\right]=1\;\forall\:i,j\in\mathcal{I}\right.\right\rangle,
\end{equation}
If $k =\infty$ then we can omit $\sigma$ and $G = \prod_{\mathcal{I}} \Z_p  =  \prod_{\mathcal{I}} \langle \rho_i \rangle$.
Thus we recover Corollary \ref{corB}  and our orientation in $\theta$ in Corollary \ref{corB} can be chosen as the cyclotomic
character $\theta \colon G \rightarrow \Z_p^{\times}$.

In the above theorem, we determined $G^{[n]}$ quotients of $G_F(p)$ if $F$ is $p$-rigid. If $k < \infty$ the described action is trivial
iff $n \le p^k+1$. Thus we obtain the following interesting corollary.

\begin{cor}
Suppose that $F$ is $p$-rigid field and $k < \infty$. Then $G^{[n]}$ is abelian iff $n \le p^k +1$.
\end{cor}

In Corollary \ref{cor-prigid}  we determined the structure of  $F(p)$ for $F$ a $p$-rigid field.  It is the simplest possible structure $F(p)$ can have.
Using this structure, we determine in Theorem \ref{prigid}, and the discussion after it, that if $d(G) \ge 1$, then $G_F(p)$ fits in the exact sequence
\[ 1 \rightarrow A \rightarrow G_F(p) \rightarrow \Z_p \rightarrow 1\]
where $A$ is a topological product of copies of $\Z_p$. Therefore, if we assume that
\[ F(p) = \bigcup_{n \ge 1} F( \zeta_{p^{k+n}}, a_i^{1/p^n} , i \in \mathcal{I})\]
for some $\mathcal{I}$ using \cite{ware} Theorem 1(b), we obtain the following refinement of Corollary \ref{cor-prigid}.

\begin{cor} \label{cor-prigid-characterization}
Let $F$ be a field. Then $F$ is $p$-rigid if and only if
\[ F(p) = \bigcup_{n \ge 1} F( \zeta_{p^{k+n}}, a_i^{1/p^n} , i \in \mathcal{I}).\]
\end{cor}

\begin{rem} \label{alternativeproofs}
Note that our proofs were done using purely Galois field theoretic methods.  However, some of these arguments can be obtained by using
the theory of uniform pro-$p$ groups. (See beginning of Section \ref{subsec:3.1pwflgps} for the definition of uniform pro-$p$ groups and recall from Section \ref{subsec:4.1proof} that  $\lambda_n(G) = G^{(n)}$.)  Although this theory is worked out in \cite{analytic} only for finitely generated pro-$p$ groups, the techniques and methods and can also be extended to our groups, i.e.,  groups of the form $G = G_F(p)$. In fact, because the structure of $G$ as described in Theorem  \ref{prigid}  is very simple, these results can be proved in a straightforward manner. Here we will reformulate some of these results which are inspired by the theory of uniform pro-$p$ groups.

First of all, from Theorem \ref{lem:moduloinfinitelygeneratedcase} and Theorem \ref{prigid} we see  that if $G$ is finitely generated then $G$ is a uniform pro-$p$ group.  Observe that $G^{(2)} = \lambda_2(G) = G^p.$ This follows from the fact that the commutators in $G$ are $p$th powers. In fact, by induction on $n$, we see that $G^{(n)}= G^{p^{n-1}}$ for all $n \ge 2$. Thus we see again slightly differently the validity of  Remark \ref{rem:prigid}. The following fact is a consequence of the group structure of $G$, and in fact in the case of finitely generated pro-$p$ groups, it holds for every uniform pro-$p$ group. We omit a straightforward direct proof.

\begin{fact}
 The $p^n$th power map $G\rightarrow G^{p^n}$ induces an isomorphism of (finite) $p$-groups
\begin{displaymath}
 \xymatrix{G/G^p=G/G^{(2)} \ar[rr]^-{p^n} && G^{p^n}/G^{p^{n+1}}=G^{(n+1)}/G^{(n+2)}.}
\end{displaymath}
for every $n>1$.
\end{fact}

By duality the above map induces the following commutative diagram
\begin{displaymath}
 \xymatrix{ H^1\left(G^{(n)},\F_p\right)\ar[rr]^-{(p^{n-1})^\vee}\ar@{-}[d]^\wr && H^1\left(G,\F_p\right)\ar@{-}[d]^\wr \\
 \dot{F}^{(n)}/(\dot{F}^{(n)})^p && \dot{F}/\dot{F}^p\ar[ll]_-{\psi_n}}
\end{displaymath}
where the upper arrow is the dual of the $p^n$th power map -- and therefore $(p^{n-1})^\vee$ is an isomorphism --
and the vertical arrows are the Kummer isomorphisms.
Consequently also $\dot{F}/\dot{F}^p$ and $\dot{F}^{(n)}/(\dot{F}^{(n)})^p$
are isomorphic as $\F_p$-vector spaces.
In particular,
\[ \psi_n\left([\zeta_{p^k}]_F\right)=\left[\zeta_{p^{k+n-1}}\right]_{F^{(n)}}\quad\text{and}
\quad \psi_n\left([a_i]_F\right)=\left[a_i^{1/p^{n-1}}\right]_{F^{(n)}}\]
for every $n>1$ and $i\in\mathcal{I}$. This last conclusion is consistent with Theorem \ref{lem:moduloinfinitelygeneratedcase} and Remark \ref{complicatedway}.

\end{rem}

From Theorem~A is possible now to sort out the following new characterization of $p$-rigidity
which restricts to Galois groups of finite exponent.

\begin{cor}\label{thm:centre}
The field $F$ is $p$-rigid if and only if one has
\[\Gal\left(F^{(2)}/F^{\{3\}}\right)\subseteq\Zen\left(\Gal\left(F^{\{3\}}/F\right)\right).\]
\end{cor}

\begin{proof}
Recall first that
\begin{equation}\label{eq:corollarycentre}
 \Gal\left(F^{\{3\}}/F\right)=\frac{G}{G^{\{3\}}}\quad\text{and}
\quad\Gal\left(F^{(2)}/F^{\{3\}}\right)=\frac{G^{(2)}}{G^{\{3\}}}.
\end{equation}

Assume that $F$ is $p$-rigid.
Then by Theorem~A one has $F^{\{3\}}=F^{(3)}$.  By the construction of $F^{(3)}$, we see that $\Gal(F^{(2)}/F^{(3)}) = \Gal(F^{(2)}/F^{\{3\}})$ is the central subgroup of $\Gal(F^{\{3\}}/F)$.

Conversely, assume that $\Gal(F^{\{3\}}/F^{(2)})$ is central in $\Gal(F^{\{3\}}/F)$.
By (\ref{eq:corollarycentre}), this implies that the commutator subgroup $[G,G^{(2)}]$ is contained in $G^{\{3\}}$.
Since
\[G^{\{3\}}=\Phi\left(G^{(2)}\right)\geq\left(G^{(2)}\right)^p      \quad\text{and}
\quad G^{(3)}=\left(G^{(2)}\right)^p\left[G,G^{(2)}\right],\]
it follows that $G^{\{3\}}$ contains $G^{(3)}$, and thus $G^{\{3\}}=G^{(3)}$.
Therefore $F$  is $p$-rigid by Theorem~A.
\end{proof}

We proved the each $p$-rigid field is hereditary $p$-rigid. In particular, if $F$ is $p$-rigid, then for each finite extension $K/F$,
$K \subset F(p)$ is again $p$-rigid. Then there is a natural question of whether $F$ is $p$-rigid in the above situation when  we assume that $K$ is $p$-rigid.  If $\dot{F}/\dot{F}^p$ is finite, the answer is yes as we will show below, and the
proof is a quite remarkable use of Serre's theorem (\cite{serre65}) on cohomological dimension of open subgroups of pro-$p$ groups, a consequence of Bloch-Kato conjecture on cohomological dimensions, and an elementary observation on the growth of
$p$-power classes.

\begin{thm} \label{goingdown}
Suppose  that $F$ is any field such that $G = G_F(p)$ is finitely generated pro-$p$ group.  If there exists a finite
extension $K/F$, $K \subseteq F(p)$, such that $K$ is $p$-rigid, then so is $F$.
\end{thm}

\begin{proof}
Because $G$ is finitely generated, we see that the minimal number of generators of $G$ is equal to
\[ d:= d(G) = \dim_{\mathbb{F}_p} H^1(G, \mathbb{F}_p ) = \dim_{\mathbb{F}_p} \dot{F}/\dot{F}^p < \infty. \]
In particular, $\dot{F}/(\dot{F})^p$ is a finite group.  Since $K/F$ is a finite extension in $F(p)$, from basic Galois theory and theory of $p$-groups we see that there is chain of extensions
\[ F = K_0 \subset K_1 \subset \cdots \subset K_s = K\]
such that $[K_{i+1} \colon K_i] = p$ for each $i= 0, 1, \cdots, s-1$.
Thus by Kummer theory each $K_{i+1}$ is of the form $K_{i+1} = K_i (c_i^{1/p})$ for some $c_i \in \dot{K}$.  Now observe by induction on $i$ that
\[  \dim_{\mathbb{F}_p} \dot{F}/\dot{F}^p \le  \dim_{\mathbb{F}_p}  \dot{K_i}/\dot{K_i}^p \le \dim_{\mathbb{F}_p}  \dot{K}/\dot{K}^p.\]
Indeed, by Kummer theory we have a natural embedding
\[ \psi_i  \colon (\dot{K_i}/\dot{K_i}^p)/ \langle [c_i]_{K_i} \rangle \rightarrow \dot{K_{i+1}}/\dot{K_{i+1}}^p\]
where $ [c_i]_{K_i}$ is the element of $\dot{K_i}/\dot{K_i}^p$ corresponding to $c_i$ and $\langle [c_i]_{K_i} \rangle$ is the
subgroup of $\dot{K_i}/\dot{K_i}^p$ generated by $[c_i]_{K_i}$.
Hence
\[   \dim_{\mathbb{F}_p}  \dot{K_i}/\dot{K_i}^p \le  1 +  \dim_{\mathbb{F}_p}  (\dot{K_{i+1}}/\dot{K_{i+1}}^p). \]
However, the ``lost of $[c_i]_{K_i}$" is compensated by $[c_i^{1/p}]_{K_i}$.  Indeed in the group $\dot{K_{i+1}}/\dot{K_{i+1}}^p$, the element $[c_i^{1/p}]_{K_i}$ is independent from the $B_i : = \text{image of} \psi_i$.  This means that
\[[c_i^{1/p}]_{K_i} \cap B_i = \{ [1]_{K_{i+1}}\}. \]
Indeed $N_{K_{i+1/K_i}} (c_i^{1/p}) = c_i \notin \dot{K_i}^p$, but the set of all norms $N_{K_{i+1/K_i}}$ of elements in $B_i$,
$N_{K_{i+1/K_i}}(B_i)  = \{ [1]_{K_{i+1}}\}$.
Because we assume that $K$ is $p$-rigid we see that
\[ e:= \dim_{\mathbb{F}_p}  \dot{K}/\dot{K}^p = \text{cd} \, G_K(p).  \]
But since by \cite{beckerAS}, Satz 3 we know that $G_F(p)$ is torsion free we can conclude from  \cite{serre65} th\'eor\`eme that
\[\text{cd} \, G_K(p) = \text{cd} \, G_F(p).\]
Hence we conclude that
\[d \le e = \text{cd} \, G_F(p).\]
On the other hand, from the Bloch-Kato conjecture (now Rost-Voevodsky theorem) one can conclude that (see \cite{cem})
\[ e:= \text{cd} \, G_F(p) \le d.\]
Hence $e = d$ and from Corollary \ref{corE} we conclude that $F$ is $p$-rigid.
\end{proof}

\subsection{Fast solvability of algebraic equations:}\label{subsec:fastsolvability}
Recall from the Introduction that a {\it non-nested root} over a field
$F$ is an element $\alpha\in\bar{F}^s$ such that $\alpha=\sqrt[n]{a}$ for $a\in F$.
For example, given elements $a_i,b_i\in \dot F$, the expression
\[\sqrt[n]{a_0+a_1\sqrt[n_1]{b_1}+\ldots+a_r\sqrt[n_r]{b_r}},\]
is a nested root, if $n$ and some $n_i$ are both larger than 1.

\begin{defi}
A polynomial $f\in F[X]$ is said to be {\it fast-solvable} if it is solvable by radicals
(in the sense of Galois) and its splitting field over $F$ is contained in a field $L$ generated by non-nested roots, i.e.,
\[L=F\left(\sqrt[n_1]{a_1},\ldots,\sqrt[n_r]{a_r}\right),\quad a_i\in\dot{F},n_i>1.\]
\end{defi}

Therefore, by Corollary \ref{cor-prigid-characterization} every irreducible polynomial in $F[X]$
with splitting field of $p$-power degree is fast-solvable for $F$ a $p$-rigid field.
On the other hand, observe that Ferrari's formula for a solution of a quartic equation makes use of nested roots.
This suggests that a general quartic polynomial is not fast-solvable
in spite of the name of the author of the formula.
Following the suggestion of the referee we leave the reader a non-trivial problem
of showing that there is no ``Porsche formula''  which provides a fast solution for a general quartic equation.

\subsection{Analytic pro-$p$ group and dimension subgroups}\label{subsec:4.4dimension}
A {\it $p$-adic analytic pro-$p$ group} is a $p$-adic analytic manifold
which is also a group such that the group operations are given by analytic functions (see \cite[Ch. 8]{analytic}).
Analytic pro-$p$ groups were first introduced and studied in depth by M. Lazard \cite{lazard},
and are now an object of research, both in group theory and number theory (see \cite{horizons}).

Lazard found a beautiful group theoretic characterization of $p$-adic analytic pro-$p$ group. This is the main result of the book
\cite{analytic}. One variant of it is the following theorem.

\begin{thm} \cite[Chapter 8]{analytic} The following statements are equivalent for a topological group $G$.
\begin{enumerate}
\item $G$ is a compact $p$-adic analytic group.
\item $G$ contains an open normal uniform pro-$p$ group of finite index
\item   $G$ is a profinite group containing an open subgroup which is a pro-$p$ group of finite rank.
\end{enumerate}
\end{thm}

Using this result and Theorem \ref{goingdown} we obtain the next theorem.

\begin{thm}
Assume that $\dim_{\mathbb{F}_p}  \, \dot{F}/(\dot{F})^p < \infty$. Then  $F$ is $p$-rigid if and only if $G_F(p)$ is a $p$-adic analytic pro-$p$ group.
\end{thm}

\begin{proof}
Assume first that $\dim_{\mathbb{F}_p} \dot{F}/(\dot{F})^p < \infty$ and  $F$ is $p$-rigid. Then as we pointed out in
Remark \ref{alternativeproofs}, $G$ itself is uniform and hence by previous result $G$ is $p$-adic analytic.

Assume now that  $\dim_{\mathbb{F}_p} \dot{F}/(\dot{F})^p < \infty$  and $G_F(p)$ is $p$-adic analytic.  Then there exists
an open normal uniform subgroup $H$ of $G_F(p)$.  Let $K$ be the fixed field of $H$. Then $H = G_K(p)$.  Because $H$ is uniform, it is in particular powerful and by Proposition \ref{prop:rigidiffpowerful} we see that $K$ is $p$-rigid. Now by Theorem \ref{goingdown} we see that $F$ is $p$-rigid as well.
\end{proof}

For a (not necessarily pro-$p$) group $G$, its dimension subgroups $D_n=D_n(G)$ are defined as follows:
$D_1=G$, and for $n>1$
\[D_n=D_{\lceil n/p\rceil}^p\prod_{i+j=n}[D_i,D_j],\]
where $\lceil n/p\rceil$ is the least integer $k$ such that $pk\geq n$.
Dimension subgroups define the fastest descending series of $G$ such that $[D_i,D_j]\leq D_{i+j}$ and $D_i^p\leq D_{pi}$,
for any $i,j\in\N^*$.
Moreover, $D_n$ is the kernel of all the natural homomorphisms of $G$ into the unit group of $\kappa[G]/I^n$, i.e., $D_n\cong1+I^n$,
where $\kappa$ is any field of characteristic $p$, and $I$ is the augmentation ideal of $\kappa[G]$ \cite[\S 11.1]{analytic}.
The following formula, due to Lazard, provides an explicit description for $D_n$ (see \cite[Theorem 11.2]{analytic}):
\begin{equation}\label{lazard}
 D_n(G)=\prod_{ip^h\geq n}\gamma_i(G)^{p^h}.
\end{equation}

It is worth questioning how the dimension subgroups look when $G$ is the maximal pro-$p$ Galois group $G_F(p)$
of a $p$-rigid field $F$. This can be addressed using the following theorem.

\begin{thm}[\cite{analytic}, Thm. 11.4]
Let $G$ be a finitely generated pro-$p$ group.
Then $G$ has finite rank if and only if $D_n(G)=D_{n+1}(G)$ for some $n$.
\end{thm}

Thus, by Corollary~\ref{corF}, the above theorem also holds for  finitely generated $G_F(p)$ with $F$ a $p$-rigid field.
However, in our calculation below we can assume that $G_F(p)$ is any Galois group of the maximal $p$-extension of a $p$-rigid
field which is not necessarily finitely generated.
Let $G=G_F(p)$ for such a field and $k\in\N\cup\{\infty\}$ and $\theta_F$ be as above,
and set $N=\kernel(\theta_F)$.
Thus $N\cong\Z_p^{\mathcal{I}}$ as pro-$p$ subgroups
-- in particular, if $k=\infty$ then $N=G$.
Then by (\ref{eq:presentationwithcyclotomic}) one has
\[[G,G]=N^{p^k}\quad\text{and}\quad\gamma_i(G)=N^{p^{k(i-1)}},\]
for $i>1$. (We implicitly set $p^\infty=0$. Thus in the case $k = \infty$, all commutators are trivial and the calculation below became very simple but still giving us the same result formally as below independent of whether $k$ is finite or infinite.)
Assume $p^{\ell-1}<n\leq p^\ell$, with $\ell\geq1$, so that $\ell=\lceil\log_p(n)\rceil$,
i.e., $\ell$ is the least integer such that $\ell\geq\log_p(n)$.
Hence, by Lazard's formula (\ref{lazard}), one has
\begin{equation}\label{eq:computdimension}
D_n(G)=\gamma_1(G)^{p^\ell}\prod_{ip^h\geq n}\gamma_i(G)^{p^h}=G^{p^\ell}\prod_{ip^h\geq n} N^{p^{k(i-1)+h}},
\end{equation}
where $i\geq2$.

We shall show that for every $i,h$ such that $i\geq2$ and $ip^h\geq n$, one has the inequality
\begin{equation}\label{inequality1}
k(i-1)+h\geq\ell,
\end{equation}
so that $N^{p^{k(i-1)+h}}\leq G^{p^\ell}$ and $D_n(G)=G^{p^\ell}$.
If $h\geq\ell$, then (\ref{inequality1}) follows immediately.
Otherwise, notice that $i>p^{\ell-h-1}\geq1$, as $ip^h\geq n>p^{\ell-1}$, which implies
\begin{equation}\label{inequality2}
 k(i-1)>k\left(p^{(\ell-h)-1}-1\right).
\end{equation}
Therefore, for $\ell-h\geq2$, the inequality (\ref{inequality2}) implies $k(i-1)\geq\ell-h$,
and thus (\ref{inequality1}),
whereas for $\ell-h=1$ (\ref{inequality1}) follows from the fact that $i\geq2$.

Altogether, this shows that
\begin{eqnarray*}
&& D_n(G)=G^{p^\ell},\quad\text{with } p^{\ell -1} < n \le p^{\ell},   \\
&\text{and}& \frac{D_n(G)}{D_{n+1}(G)}\cong\left\{\begin{array}{cc} (\Z/p\Z)^{\mathcal{I}} & \text{for $n$ a $p$-power}\\
1 & \text{otherwise.} \end{array}\right.
\end{eqnarray*}

\section*{Acknowledgements}
We thank A. Adem, D. Karagueuzian, and I. Efrat
for working with the second author over the years on related topics which strongly influenced the way
we think about rigidity and its unique place in Galois theory.
Moreover, we thank E. Aljadeff, D. Neftin, J. Sonn
and N.D. Tan  for their interest and helpful suggestions on this paper.
Also the second author is grateful to J. Koenigsmann for calling his attention to Koenigsmann's paper
with A. Engler quoted and used in our paper. We are  grateful to Leslie Hallock for her passion for a proper and elegant use of English language and for her kind but forceful
advice related to its use in our paper.
We are also grateful to the referee for his/her encouragement and
for providing us a number of valuable suggestions to improve the exposition.
Last but not least, the second and third authors are very grateful to Th.S. Weigel who introduced them to each other.

\end{document}